\documentclass[12pt,leqno]{article}
\usepackage{amsfonts}
\usepackage{amsmath, amsthm, amsfonts, amssymb, color}
\usepackage{mathrsfs}
\usepackage{color}
\usepackage{stmaryrd}
\makeatletter
\newcommand{\Spvek}[2][r]{%
  \gdef\@VORNE{1}
  \left(\hskip-\arraycolsep%
    \begin{array}{#1}\vekSp@lten{#2}\end{array}%
  \hskip-\arraycolsep\right)}

\def\vekSp@lten#1{\xvekSp@lten#1;vekL@stLine;}
\def\vekL@stLine{vekL@stLine}
\def\xvekSp@lten#1;{\def\temp{#1}%
  \ifx\temp\vekL@stLine
  \else
    \ifnum\@VORNE=1\gdef\@VORNE{0}
    \else\@arraycr\fi%
    #1%
    \expandafter\xvekSp@lten
  \fi}
\makeatother

\newtheorem{thm}{Theorem}[section]

\newtheorem{lem}[thm]{Lemma}

\newtheorem{rem}[thm]{Remark}
\theoremstyle{definition}

\newcommand{\scr}[1]{\mathscr #1}
\definecolor{wco}{rgb}{0.5,0.2,0.3}

\numberwithin{equation}{section} \theoremstyle{remark}

\newcommand{\ua}{\uparrow}

\title{{\bf    Convergence Rate of  Euler-Maruyama Scheme  for SDEs with Rough Coefficients} }
\author{
{\bf  Jianhai Bao$^{a)}$,  Xing Huang$^{b)}$, Chenggui Yuan$^{c)}$}\\
\footnotesize{$^{a)}$School of Mathematics and Statistics, Central
South
University, Changsha 410083, China}\\
\footnotesize{$^{b)}$School of Mathematical Sciences, Beijing Normal
University, Beijing 100875, China}\\
\footnotesize{$^{c)}$Department of Mathematics, Swansea University,
Singleton Park, SA2 8PP, UK}}
\begin{document}
\allowdisplaybreaks
\def\R{\mathbb R}  \def\ff{\frac} \def\ss{\sqrt} \def\B{\mathbf
B}
\def\N{\mathbb N} \def\kk{\kappa} \def\m{{\bf m}}
\def\ee{\varepsilon}\def\ddd{D^*}
\def\dd{\delta} \def\DD{\Delta} \def\vv{\varepsilon} \def\rr{\rho}
\def\<{\langle} \def\>{\rangle} \def\GG{\Gamma} \def\gg{\gamma}
  \def\nn{\nabla} \def\pp{\partial} \def\E{\mathbb E}
\def\d{\text{\rm{d}}} \def\bb{\beta} \def\aa{\alpha} \def\D{\scr D}
  \def\si{\sigma} \def\ess{\text{\rm{ess}}}
\def\beg{\begin} \def\beq{\begin{equation}}  \def\F{\scr F}
\def\Ric{\text{\rm{Ric}}} \def\Hess{\text{\rm{Hess}}}
\def\e{\text{\rm{e}}} \def\ua{\underline a} \def\OO{\Omega}  \def\oo{\omega}
 \def\tt{\tilde} \def\Ric{\text{\rm{Ric}}}
\def\cut{\text{\rm{cut}}} \def\P{\mathbb P} \def\ifn{I_n(f^{\bigotimes n})}
\def\C{\scr C}   \def\G{\scr G}   \def\aaa{\mathbf{r}}     \def\r{r}
\def\gap{\text{\rm{gap}}} \def\prr{\pi_{{\bf m},\varrho}}  \def\r{\mathbf r}
\def\Z{\mathbb Z} \def\vrr{\varrho} \def\ll{\lambda}
\def\L{\scr L}\def\Tt{\tt} \def\TT{\tt}\def\II{\mathbb I}
\def\i{{\rm in}}\def\Sect{{\rm Sect}}  \def\H{\mathbb H}
\def\M{\scr M}\def\Q{\mathbb Q} \def\texto{\text{o}} \def\LL{\Lambda}
\def\Rank{{\rm Rank}} \def\B{\scr B} \def\i{{\rm i}} \def\HR{\hat{\R}^d}
\def\to{\rightarrow}\def\l{\ell}\def\iint{\int}
\def\EE{\scr E}\def\no{\nonumber}
\def\A{\scr A}\def\V{\mathbb V}\def\osc{{\rm osc}}
\def\BB{\scr B}\def\Ent{{\rm Ent}}\def\3{\triangle}
\def\U{\scr U}\def\8{\infty}\def\1{\lesssim}\def\HH{\mathrm{H}}

\maketitle

\begin{abstract}
In this paper, we are concerned with convergence rate of
Euler-Maruyama scheme for stochastic differential equations with
rough coefficients. The key contributions lie  in (i), by means of
  regularity of non-degenerate Kolmogrov equation, we investigate
convergence rate of Euler-Maruyama scheme for a class of stochastic
differential equations, which allow the drifts to be Dini-continuous
and unbounded; (ii) by the aid of regularization properties of
degenerate Kolmogrov equation, we discuss convergence rate of
Euler-Maruyama scheme for a range of degenerate stochastic
differential equations, where the drift is locally H\"older-Dini
continuous of order $\ff{2}{3}$ with respect to the first component,
and is merely Dini-continuous concerning the second component.

\end{abstract} \noindent
 \vskip 2cm

\section{Introduction and Main Results}
It is well-known that convergence rate of Euler-Maruyama (EM) for
stochastic differential equations (SDEs) with regular coefficients
is one-half, see, e.g., \cite{Klo,Mao}. With regard to convergence
rate of EM scheme under various settings, we refer to, e.g.,
\cite{BY} for stochastic differential delay equations (SDDEs)  with
polynomial growth with respect to (w.r.t.) the delay variables,
\cite{GS} for SDDEs under local Lipschitz and also under
monotonicity condition, \cite{G98,HK,NT2} for SDEs with
discontinuous coefficients, and \cite{YM} for SDEs under
log-Lipschitz condition.

Recently, convergence rate of EM scheme for SDEs with irregular
coefficients has also gained much attention. For instance,
by  the Meyer--Tanaka formula, \cite{Y02} revealed convergence rate
in $L^1$-norm sense  for a range of  SDEs, where the drift term is
Lipschitzian and the diffusion term is H\"older continuous w.r.t.
spatial variable; Adopting the Yamada-Watanabe approximation
approach,  \cite{GR} extended \cite{Y02} to discuss strong
convergence rate  in $L^p$-norm sense; Using the Yamada-Watanabe
approximation trick and heat kernel estimate,  \cite{NT} studied
strong convergence rate in $L^1$-norm sense for a class of
non-degenerate SDEs, where the bounded drift term satisfies a weak
monotonicity and is of bounded variation w.r.t.  a Gaussian measure
and the diffusion term is H\"older continuous; Applying the Zvonkin
transformation, \cite{PT} discussed strong convergence rate in
$L^p$-norm sense  for SDEs with additive noise, where the  drift
coefficient is bounded and  H\"older continuous.

It is worth  pointing out that \cite{NT,PT} focused on convergence
rate of EM for SDEs with {\it H\"older continuous and bounded
drift}, which, nevertheless, rules out some interesting examples. On
the other hand, most of the existing literature on convergence rate
of EM scheme is concerned with {\it non-degenerate SDEs},  the
corresponding issue for {\it degenerate SDEs} is scarce. So, in this
work, our goal is to discuss convergence rate of EM method for SDEs
with rough coefficients, which may allow  SDEs  involved to be
degenerate.  For wellposedness of (path-dependent) SDEs with
singular coefficients, we refer to, e.g., \cite{B16, W,WZ,WZ1} for
more details.

Throughout the paper, the following notation will be used.  Let
$\|\cdot\|$ and $\|\cdot\|_{\mathrm{HS}}$ stand for the usual
operator norm and the Hilbert-Schmidt norm, respectively. Fix $T>0$
and   set $\|f\|_{T,\8}:=\sup_{t\in[0,T],x\in \R^m}\|f(t,x)\|$ for
an operator-valued map $f$ on $[0,T]\times \R^m$. Denote
$\mathbb{M}_{\rm non}^n$ by the collection of all nonsingular
$n\times n$-matrices and $\mathbb{M}_{\rm non}^{n,cc}$ by a closed
and convex subset of $\mathbb{M}_{\rm non}^n$.  Let $\mathscr{S}_0$
be the collection of all slowly varying functions
$\phi:\R_+\mapsto\R_+$  at zero in Karamata's sense (i.e.,
$\lim_{t\rightarrow0}\frac{\phi(\lambda t)}{\phi(t)}=1$ for any
$\ll>0$), which are bounded from $0$ and $\8$ on $[\vv,\8)$ for any
$\vv>0$. For more properties of slowly varying functions, we refer
to, e.g., Bingham et al. \cite{BGT}. Let $\mathscr{D}_0$ be the
family of Dini functions, i.e., \beg{equation*}\beg{split} \D_0:&=
\Big\{\phi\Big|\phi: \R_+\mapsto \R_+ \mbox{ is increasing and }
\int_0^1{\frac{\phi(s)}{s}\d s}<\infty\Big\}.
\end{split}\end{equation*}
A function $f:\R^m\mapsto\R^n$ is called Dini-continuity if there
exists $\phi\in\D_0$ such that $|f(x)-f(y)|\le\phi(|x-y|)$ for any
$x,y\in\R^m.$ Remark that every Dini-continuous function is
continuous and every Lipschitz continuous function is
Dini-continuous; Moreover, if $f$ is H\"older
continuous, 
then $f$ is Dini-continuous. Nevertheless, there are numerous
Dini-continuous functions, which are not H\"older continuous at all;
see, e.g., $\varphi(x)=(\log(c+x^{-1}))^{-(1+\dd)}, x>0$, for some
constants $\dd>0$ and $c\ge\e^{3+2\dd},$ and $\varphi(x)=0$ for $x=0.$ For
some sufficiently small $\vv\in(0,1)$, set
 \beg{equation*}\beg{split} \D:= \{\phi\in\D_0| \phi^{2} \text{ is
 concave}\}~~~\mbox{ and }~~~\D^\vv:= \{\phi\in\D|\phi^{2(1+\vv)} \text{ is
 concave} \}.
\end{split}\end{equation*}
Clearly, $\varphi$ constructed above belongs to $\D\cap\D^\vv.$ A
function $f:\R^m\mapsto\R^n$ is called H\"older-Dini continuity of
order $\aa\in[0,1)$ if
\begin{equation*}
|f(x)-f(y)|\le |x-y|^\alpha\phi(|x-y|),~~~~~|x-y|\le1
\end{equation*}
for some $\phi\in\D_0.$ For any measurable function
$\psi:(0,1]\mapsto\R_+$ and $f\in C(\R^n)$, let
\begin{equation*}
[f]_{\psi}=\sup_{|x-y|\le1}\frac{|f(x)-f(y)|}{\psi(|x-y|)},
~~\|f\|_\8=\sup_{x\in\R^n}|f(x)| ~\mbox{ and }~\interleave
f\interleave_\psi=[f]_{\psi}+\|f\|_\8.
\end{equation*}
For notational simplicity, we shall write $\R^{2n}$ instead of
$\R^n\times \R^n$. For any measurable functions
$\psi_1,\psi_2:(0,1]\mapsto\R_+$ and $f\in C(\R^{2n})$, let
\begin{equation*}
\begin{split}
[f]_{\psi_1,\8}&=\sup_{x^{(2)}\in\R^n}[f(\cdot,x^{(2)})]_{\psi_1},~~~~[f]_{\8,\psi_2}=\sup_{x^{(1)}\in\R^n}[f(x^{(1)},\cdot)]_{\psi_2},\\
[f]_{\psi_1,\psi_2}&=[f]_{\psi_1,\8}+[f]_{\8,\psi_2},~~~~~\interleave
f\interleave_{\psi_1,\psi_2}=[f]_{\psi_1,\psi_2}+\|f\|_\8,\\
 \interleave f\interleave_{\psi_1,\8}&=[f]_{\psi_1,\8}+\|f\|_\8,~~~~~~~\interleave
 f\interleave_{\8,\psi_2}=[f]_{\8,\psi_2}+\|f\|_\8.
\end{split}
\end{equation*}
Write the gradient operator on $\R^{2n}$ as
$\nn=(\nn^{(1)},\nn^{(2)})$, where $\nn^{(1)}$ and $ \nn^{(2)}$
stand for the gradient operators for the first and the second
components, respectively.

Before proceeding further, a few words about the notation are in
order. Generic constants will be denoted by $c$; we use the
shorthand notation $a\1b$ to mean $a\le cb$. If the constant $c$
depends on a parameter $p$, we shall also write $c_p$ and $a\1_p b$.

\subsection{Non-degenerate SDEs with Bounded
Coefficients}\label{sec1.1} In this subsection, we consider an SDE
on $(\R^n,\<\cdot,\cdot\>, |\cdot|)$
 \beq\label{1.1} \d X_t=b_t(X_t)\d t+\sigma_t(X_t)\d
W_t, ~~~t>0,~~~X_0=x,
\end{equation}
where $b: \R_+\times\R^n\mapsto\R^n$, $\si:
\R_+\times\R^n\mapsto\R^n\otimes \R^n$, and $(W_t)_{t\geq 0}$ is an
$n$-dimensional Brownian motion on a complete filtered probability
space $(\OO, \F, (\F_t)_{t\ge 0}, \P)$.

\smallskip
With regard to \eqref{1.1},
 we suppose that  there exists $\phi\in\D$ such that for any
 $s,t\in[0,T]$ and $x,y\in\R^n,$
 \beg{enumerate}
\item[\bf{(A1)}]   $\si_t\in C^2(\R^n;\R^n\otimes \R^n)$,
$\si_t\in\mathbb{M}_{\rm non}^n$,  and
\begin{equation}\label{bao3} \|b\|_{T,\8}+\sum_{i=0}^{2}\|\nabla^{i}
\sigma\|_{T,\infty}+\|\nn\sigma^{-1}\|_{T,\infty}+\|\sigma^{-1}\|_{T,\infty}<\8,
\end{equation}
where $\nn^i$ means the $i$-th order gradient operator;

\item[\bf{(A2)}] (Regularity of $b$ w.r.t.   spatial variables)
\begin{equation*}
|b_t(x)-b_t(y)|\leq \phi(|x-y|);
\end{equation*}
\item[\bf{(A3)}] (Regularity of $b$ and $\si$ w.r.t.   time variables)
\begin{equation*}
|b_{s}(x)-b_{t}(x)|+\|\sigma_{s}(x)-\sigma_{t}(x)\|_{\mathrm{HS}}\leq
\phi(|s-t|). \end{equation*}

\end{enumerate}
Under {\bf(A1)} and {\bf(A2)}, \eqref{1.1} admits a unique
non-explosive strong solution $(X_t)_{t\in[0,T]}$; see, e.g.,
\cite[Theorem 1.1]{W}.

\smallskip

Without loss of generality, we take an integer  $N>0$ sufficiently large such
that the stepsize $\delta:=T/N\in(0,1)$. 
The continuous-time EM scheme corresponding to
  \eqref{1.1} is
\begin{equation}\label{5eq20}
\d Y_t=b_{t_\delta}( Y_{t_\delta})\d t+\sigma_{t_\delta} (
Y_{t_\delta})\d W_t,~~~t>0,~~~Y_0=X_0=x.
\end{equation}
Herein, $t_\delta:=\lfloor t/\delta\rfloor\delta$ with $\lfloor
t/\delta\rfloor$ being the integer part of $t/\delta$.

\smallskip

The first contribution in this paper is stated as follows.

\beg{thm}\label{T1.1} {\rm Under {\bf (A1)}-{\bf (A3)},
\begin{equation*}
 \E\Big(\sup_{0\le t\le
T}|X_t-Y_t|^2\Big)\1_T\phi(C_T\ss\dd)^2
\end{equation*}
for some constant $C_T\ge1$.
 }
\end{thm}

\begin{rem}
{\rm In Theorem \ref{T1.1}, taking $\phi(x)=x^\beta$ for $x>0$ and
$\beta\in(0,1/2],$ we arrive at
\begin{equation*}
 \E\Big(\sup_{0\le t\le
T}|X_t-Y_t|^2\Big) \1_T\dd^\beta,
\end{equation*}
which covers   \cite[Theorem 2.13]{PT} with $\beta\in(0,1/2]$
therein. On the other hand, by going carefully through the argument
of Theorem \ref{T1.1}, we can allow the H\"older exponent
$\bb\in(1/2,1)$ whenever we consider the error bound of
$\E\Big(\sup_{0\le t\le T}|X_t-Y_t|^p\Big)$ for any $p\in[1,1/\bb]$.
Moreover, choosing $\phi(x)=x,x\ge0,$ and inspecting closely the
argument of Theorem \ref{T1.1}, one has
\begin{equation*}
 \E\Big(\sup_{0\le t\le
T}|X_t-Y_t|^2\Big) \1_T\dd,
\end{equation*}
which reduces to the classical result on  strong convergence of EM
scheme for SDEs with regular coefficients, see, e.g., \cite{Klo}.

}
\end{rem}

\subsection{Non-degenerate SDEs with Unbounded
Coefficients}\label{sec1.2}
In Theorem \ref{T1.1},   the
coefficients are uniformly bounded, and that the drift term $b$
satisfies the global Dini-continuous condition (see {\bf(A2)}
above), which
 seems to be a little bit stringent. Therefore, it is quite
natural to replace uniform boundedness by local boundedness and
global Dini continuity by local Dini continuity, respectively.

\smallskip

In lieu of {\bf (A1)}-{\bf (A3)}, concerning  \eqref{1.1} we assume
that for any $s,t\in[0,T]$ and $k\ge1$,
 \beg{enumerate}
\item[{\bf
(A1')}] $\si_t\in C^2(\R^n;\R^n\otimes \R^n)$,
$\si_t\in\mathbb{M}_{\rm non}^n$, and
\beg{equation*}\beg{split} |b_t(x)|+\sum_{i=0}^{2}\|\nabla^{i}
\sigma_t(x)\|_{\mathrm{HS}}+\|\nn\sigma_t^{-1}(x)\|_{\mathrm{HS}}+\|\sigma_t^{-1}(x)\|_{\mathrm{HS}}\leq
K_T (1+|x|),~~~x\in\R^n
\end{split}\end{equation*}
for some constant $K_T>0$;
\item[{\bf
(A2')}] (Regularity of $b$ w.r.t.  spatial variables)  There exists
$\phi_k\in\D$ such that
\begin{equation*}|b_{t}(x)-b_{t}(y)|\leq \phi_k(|x-y|),~~~|x|\vee|y|\leq k;
\end{equation*}
\item[{\bf
(A3')}](Regularity of $b$ and $\si$ w.r.t.   time variables) For
$\phi_k\in\D$  such that {\bf (A2')},
\begin{equation*}
|b_{s}(x)-b_{t}(x)|+\|\sigma_{s}(x)-\sigma_{t}(x)\|_{\rm HS}\leq
\phi_k(|s-t|),~~~~~|x| \leq k.
\end{equation*}
\end{enumerate}

By the cut-off approach, Theorem \ref{T1.1} can be generalized to
cover SDEs with local Dini-continuous coefficients, which is
presented as below.

 \beg{thm}\label{T1.2} {\rm Under {\bf (A1')}-{\bf (A3')},
\begin{equation}\label{h6}
\lim_{\dd\rightarrow0} \E\Big(\sup_{0\le t\le T}|X_t-Y_t|^2\Big)=0.
\end{equation}
Moreover, if  $\phi_k(s)=\e^{\e^{c_0k^4}}s^{\alpha}, s\ge0,$ for
some $\alpha\in(0,1/2]$ and   $c_0>0$, then
\begin{equation}\label{z7}
\E\Big(\sup_{0\le t\le T}|X_t-Y_t|^2\Big)\1\inf_{\vv\in(0,1)}\Big\{
(\log\log(\dd^{-\aa\vv}))^{-\frac{1}{4}}+\dd^{\alpha(1-\vv)}\Big\}.
\end{equation}
}
\end{thm}

\begin{rem}
{\rm For the case of bounded and H\"older continuous drift $b\in
L^1(\R)$ and bounded diffusion coefficient $\si$, \cite[Theorem
2.6]{NT1} studied convergence rate of EM scheme for a class of
 scalar non-degenerate SDEs.  By a cut-off approach, \cite[Theorem
2.7]{NT1} extended \cite[Theorem 2.6]{NT1} to the case that  the
drift and diffusion terms are bounded. While, Theorem \ref{T1.2}
reveals convergence rate of EM scheme for SDEs with rough
coefficients, which allows the drift term to be unbounded and
H\"older continuous. }
\end{rem}

\subsection{Degenerate SDEs}\label{sec1.3}
 So far, most of the existing literature on convergence of EM scheme for
 SDEs with irregular coefficients is concerned with  non-degenerate
 SDEs; see, e.g., \cite{NT,NT1,PT} for SDEs driven by Brownian motions, and \cite{MX,PT} for SDEs driven by jump processes. The issue for the setup of degenerate SDEs has not yet been
considered to date to the best of our knowledge. Nevertheless, in
this subsection, we make an attempt to discuss the topic for
 degenerate SDEs with rough coefficients.

 Consider the following degenerate SDE on $\R^{2n}$
\begin{equation}\label{b3}
\begin{cases}
\d X_t^{(1)}=b^{(1)}_t(X_t^{(1)},X_t^{(2)})\d t,~~~~~~~~~~~~~~~~~~~~~~~~~~~~~~~~~X_0^{(1)}=x^{(1)}\in\R^n,\\
\d X_t^{(2)}=b^{(2)}_t(X_t^{(1)},X_t^{(2)})\d
t+\si_t(X_t^{(1)},X_t^{(2)})\d W_t,~~~~~~~X_0^{(2)}=x^{(2)}\in\R^n,
\end{cases}
\end{equation}
where $b^{(1)}_t,b^{(2)}_t:\R^{2n}\mapsto\R^{n}$, $\si_t: \R^{2n}
\mapsto\R^n\otimes\R^n$,  and $(W_t)_{t\ge0}$ is an $n$-dimensional
Brownian motion defined on a filtered probability space $(\OO, \F,
(\F_t)_{t\ge 0}, \P)$.
\eqref{b3} is also called a stochastic Hamiltonian system, which has
been
 investigated extensively in \cite{GW,WZ13,Zhang} on Bismut formulae,  in
\cite{MSH} on ergodicity,  in \cite{W14} on hypercontractivity, and
in \cite{C12,WZ,WZ1} on wellposedness, to name a few.

For any $x=(x^{(1)},x^{(2)}),y=(y^{(1)},y^{(2)})\in\R^{2n}$ and
$t\in[0,T]$,  assume that there exists
$\phi\in\mathscr{D}^\vv\cap\mathscr{S}_0$  such that
 \beg{enumerate}
\item[{\bf(C1)}] (Hypoellipticity) $\nn^{(2)}b_t^{(1)} \in\mathbb{M}_{\rm
non}^{n,cc}$, $\si_t(x)\in \mathbb{M}_{\rm non}^{n }$, and
\begin{equation*}
\|b^{(1)}\|_{T,\8}+\|b^{(2)}\|_{T,\8}+
\|\nn^{(2)}b^{(1)}\|_{T,\8}+\sum_{i=0}^2\|\nn^i\si\|_{T,\8}+\|\si^{-1}\|_{T,\8}<\8;
\end{equation*}
\item[{\bf(C2)}]  (Regularity of $b^{(1)}$ w.r.t. spatial variables)
\begin{equation*}
\begin{split}
&|b_t^{(1)}(x)-b_t^{(1)}(y)|\le
|x^{(1)}-y^{(1)}|^{\ff{2}{3}}\phi(|x^{(1)}-y^{(1)}|)~~~~~~\mbox{ if
}x^{(2)}=y^{(2)},\\
&\|\nn^{(2)}b_t^{(1)}(x)-\nn^{(2)}b_t^{(1)}(y)\|_{\mathrm{HS}}\le
 \phi(|x^{(2)}-y^{(2)}|)~~~~~~~\mbox{ if }x^{(1)}=y^{(1)};
\end{split}
\end{equation*}

\item[{\bf(C3)}]  (Regularity of $b^{(2)}$ w.r.t. spatial variables)
\begin{equation*}
|b_t^{(2)}(x)-b_t^{(2)}(y)|\le
|x^{(1)}-y^{(1)}|^{\ff{2}{3}}\phi(|x^{(1)}-y^{(1)}|)+
\phi^{\ff{7}{2}}(|x^{(2)}-y^{(2)}|);
\end{equation*}
\item[{\bf(C4)}]  (Regularity of $b^{(1)}, b^{(2)} $ and $\si$ w.r.t. time
variables)
\begin{equation*}
|b_t^{(1)}(x)-b_s^{(1)}(x)| +|b_t^{(2)}(x)-b_s^{(2)}(x)|
+\|\si_t(x)-\si_s(x)\|_{\mathrm{HS}}\le \phi(|t-s|).
\end{equation*}

\end{enumerate}

 Observe from {\bf(C2)} and {\bf(C3)} that
$b^{(1)}(\cdot,x^{(2)})$ and $b^{(2)}(\cdot,x^{(2)})$ with fixed
$x^{(2)}$ are locally H\"older-Dini continuous of order
$\frac{2}{3}$, and $\nn^{(2)}b^{(1)}(x^{(1)},\cdot)$ and
$b^{(2)}(x^{(1)},\cdot)$ with fixed $x^{(1)}$ are merely Dini
continuous. According to \cite[Theorem 1.2]{WZ},  \eqref{b3} admits
a unique strong solution under the assumptions {\bf(C1)}-{\bf(C3)}.
In fact, \eqref{b3} is wellposed under {\bf(C1)}-{\bf(C3)} with
$\phi\in\D_0\cap\mathscr{S}_0$ in lieu of
$\phi\in\D^\vv\cap\mathscr{S}_0$. Nevertheless, the requirement
$\phi\in\D^\vv\cap\mathscr{S}_0$ is imposed in order to reveal the
order of convergence for the EM scheme below.

The  continuous-time EM scheme associated with \eqref{b3} is as
follows:
\begin{equation}\label{s0}
\begin{cases}
\d Y_t^{(1)}=b^{(1)}_{t_\dd}(Y_{t_\dd}^{(1)},Y_{t_\dd}^{(2)})\d t,~~~~~~~~~~~~~~~~~~~~~~~~~~~~~~~~~X_0^{(1)}=x^{(1)}\in\R^n,\\
\d Y_t^{(2)}=b^{(2)}_{t_\dd}(Y_{t_\dd}^{(1)},Y_{t_\dd}^{(2)})\d
t+\si_{t_\dd}(Y_{t_\dd}^{(1)},Y_{t_\dd}^{(2)})\d
W_t,~~~~~~X_0^{(2)}=x^{(2)}\in\R^n,
\end{cases}
\end{equation}
where   $t_\dd$ is defined as in \eqref{5eq20}.

Another contribution in this paper reads as below.
\begin{thm}\label{th1}
{\rm Under {\bf(C1)}-{\bf(C4)},
\begin{equation*}
 \E\Big(\sup_{0\le t\le
T}|X_t-Y_t|^2\Big)\1_T\phi(C_T\ss\dd)^2
\end{equation*}
for some constant $C_T\ge1$, in which
\[
X_t:=\Spvek[c]{X_t^{(1)};X_t^{(2)}} \mbox{ and }
Y_t:=\Spvek[c]{Y_t^{(1)};Y_t^{(2)}}.
\]

}
\end{thm}

\begin{rem}
{\rm By applying the cut-off approach and refining the argument of
\cite[Theorem 2.3]{WZ}, the boundedness of coefficients can be
removed. We herein do not go into details since the corresponding
trick is quite similar to the proof of Theorem \ref{T1.2}. }
\end{rem}

The outline of this paper is organized as follows: In Section
\ref{sec2}, we  elaborate   regularity of  nondegenerate Kolmogorov
equation, which plays an important role in dealing with convergence
rate of EM scheme for nondegenerate SDEs with rough  and unbounded
coefficients; In Sections \ref{sec3}, \ref{sec4} and \ref{sec5}, we
complete the proofs of Theorems \ref{T1.1}, \ref{T1.2} and
\ref{th1}, respectively.

\section{Regularity of  Non-degenerate Kolmogorov Equation}\label{sec2}
Let $(e_i)_{i\ge1}$ be an  orthogonal basis of $\R^n.$ For any
$\lambda
>0$, consider the following $\R^n$-valued parabolic equation: \beq\label{2.1}\beg{split}
\partial_tu_t^\ll+L_tu_t^\ll+b_t+\nabla_{b_t}u_t^\ll=\lambda
u_t^\ll, \quad u_T^\ll={\bf0_n},
\end{split}\end{equation}
where  ${\bf0_n}$ is the zero vector in $\R^n$ and \beg{equation*}
L_{t}:= \frac{1}{2}\sum_{i,j}{\langle
(\si_t\si_t^*)(\cdot)e_{i},e_{j}\rangle}\nabla_{e_{i}}\nabla_{e_{j}}
 \end{equation*}
with $\si_t^*$ standing for the transpose of $\si_t.$ By solving the
corresponding coupled forward-backward SDE, one has
\begin{equation}\label{b2}
u_s^\ll=\int_s^T\e^{-\ll(
t-s)}P_{s,t}^0\{b_t+\nabla_{b_t}u_t^\ll\}\d t,
\end{equation}
where the semigroup $(P_{s,t}^0)_{0\le s\le t}$ is generated by
$(Z_t^{s,x})_{0\le s\le t}$ which solves the SDE below
\begin{equation}\label{b0}
\d Z_t^{s,x}=\sigma(Z_t^{s,x})\d W_t, ~~~t>s,~~~Z_s^{s,x}=x.
\end{equation}
For notational simplicity,  let
\begin{equation}\label{z0}
\Lambda_{T,\si}=\e^{\ff{T}{2}\|\nabla\sigma\|_{T,\infty}^2}\|\si^{-1}\|_{T,\8}£¬
\end{equation}
and
\begin{equation}\label{z2}
\begin{split}
\tt\Lambda_{T,\si}&=48\e^{288\,T^2\|\nabla\sigma\|_{T,\infty}^4}\Big\{6\ss2\e^{T\|\nabla\sigma\|_{T,\infty}^2}\|\si^{-1}\|_{T,\8}^4+T\|\nn\si^{-1}\|_{T,\8}^2\\
 &\quad+2T^2\|\nabla^2\sigma\|_{T,\infty}^2\|\si^{-1}\|_{T,\8}^2
\e^{2T\|\nabla\sigma\|_{T,\infty}^2} \Big\}.
\end{split}
\end{equation}
Moreover, set
\begin{equation}\label{z3}
\Upsilon_{T,\si}:=\ss{\tt\Lambda_{T,\si}}\Big\{3+2\|b\|_{T,\8}+
 28\Big(\Lambda_{T,\si}+
 \ss{\tt\Lambda_{T,\si}}\Big)\|b\|_{T,\8}^2\Big\}.
\end{equation}

The  lemma below plays a crucial role in investigating our numerical schemes.
\begin{lem}\label{L2.1}
{\rm Under{\bf (A1)} and {\bf (A2)}, for any $ \lambda\ge
9\pi\Lambda_{T,\si}^2\|b\|_{T,\8}^2+4(\|b\|_{T,\8}+\Lambda_{T,\si})^2,
$
\begin{enumerate}
\item[\bf{(i)}]  \eqref{2.1} (i.e., \eqref{b2}) enjoys a unique
strong solution $u^\lambda\in C([0,T];C_b^1(\R^n;\R^n)$;

\item[\bf{(ii)}]
$\|\nn u^\lambda\|_{T,\8}\le \frac{1}{2}$;

\item[\bf{(iii)}]
$ \|\nn^2
u^\lambda\|_{T,\8}\le\Upsilon_{T,\si}\int_0^T\frac{\e^{-\lambda t}
}{t}\tt\phi(\|\si\|_{T,\8}\ss t )\d t, $ where $
\tt\phi(s):=\ss{\phi^2(s)+s},\,s\ge0. $

\end{enumerate}

 }
\end{lem}

\begin{proof}
To show {\bf(i)}-{\bf(iii)}, it boils down to refine the argument of
\cite[Lemma 2.1]{W}. {\bf(i)} holds for any
$\lambda\ge4(\|b\|_{T,\8}+\Lambda_{T,\si})^2$ via the Banach
fixed-point theorem.

In what follows, we aim to show   {\bf(ii)} and {\bf(iii)}
one-by-one. It is easy to see from \eqref{b0} that
 \beg{equation}\label{z1}
\d\nn_\eta
Z_t^{s,x}=(\nn_{\nn_{\eta}Z_t^{s,x}}\sigma_t)(Z_t^{s,x})\d W_t,
~~t\geq s,~~\nn_\eta Z_s^{s,x}=\eta\in\R^n.
\end{equation}
Using It\^o's isometry and Gronwall's inequality, one has
\begin{equation}\label{a2}
\mathbb{E}|\nn_{\eta}Z_t^{s,x}|^2\le
|\eta|^2\e^{T\|\nabla\sigma\|_{T,\infty}^2}.
\end{equation}
Utilizing \cite[Theorem 7.1, p.39]{Mao} and the elementary
inequality: $(a+b)^p\le2^{p-1}(a^p+b^p)$ for any $a,b>0$ and
$p\ge1$, we deduce that \beg{equation*}\begin{split}
\mathbb{E}|\nn_{\eta}Z_t^{s,x}|^4&\leq
8\Big\{|\eta|^4+36(t-s)\|\nn\si\|_{T,\infty}^4\int_s^t\mathbb{E}|\nn_{\eta}Z_{u}^{s,x}|^4\d
u\Big\},
\end{split}\end{equation*}
 which, combining with Gronwall's
inequality, yields that \beg{equation}\begin{split}\label{b1}
\mathbb{E}|\nn_{\eta}Z_t^{s,x}|^4&\leq 8
|\eta|^4\e^{288\,T^2\|\nabla\sigma\|_{T,\infty}^4}.
\end{split}\end{equation}
 Recall from \cite[(2.8)]{W} the   Bismut formula below
\begin{equation}\label{h0}
\nn_\eta
P_{s,t}^0f(x)=\E\bigg[\frac{f(Z_t^{s,x})}{t-s}\int_s^t\<\si_r^{-1}(Z_r^{s,x})\nn_\eta
Z_r^{s,x},\d W_r\>\bigg],~~~f\in\B_b(\R^n).
\end{equation}
By the Cauchy-Schwartz inequality, the It\^o isometry and
\eqref{a2}, we obtain that
\begin{equation}\label{b7}
\begin{split}
|\nn_\eta
P_{s,t}^0f|^2(x)
&\le
\Lambda_{T,\si}^2|\eta|^2\ff{P_{s,t}^{0}f^{2}(x)}{t-s},~~~f\in\B_b(\R^n),
\end{split}
\end{equation}
where $ \Lambda_{T,\si}>0$ is defined in \eqref{z0}.
So, one infers from \eqref{b2} and \eqref{b7} that
\begin{equation*}
\begin{split}
\|\nn u_s^\ll\|&\le\int_s^T\e^{-\ll( t-s)}\|\nn
P_{s,t}^0\{b_t+\nabla_{b_t}u_t^\ll\}\|\d t\\
&\le\Lambda_{T,\si}(1+\|\nn u^\lambda\|_{T,\8})\|b\|_{T,\8}
\int_0^T\frac{\e^{-\ll t}}{\ss{t}}\d t\\
&\le\,\lambda^{-\frac{1}{2}}\ss\pi\Lambda_{T,\si}\|b\|_{T,\8}(1+\|\nn
u^\lambda\|_{T,\8}).
\end{split}
\end{equation*}
Thus, {\bf(ii)} follows by taking
$\lambda\ge9\pi\Lambda_{T,\si}^2\|b\|_{T,\8}^2$.

In the sequel, we intend to verify {\bf(iii)}. Set
$\gamma_{s,t}:=\nabla_{\eta}\nabla_{\eta^{\prime}}Z_t^{s,x}$ for any
$\eta,\eta'\in\R^n$. Notice from \eqref{z1} that \begin{equation*}\d
\gamma_{s,t}=\big\{\nabla_{\gamma_{s,t}}\sigma_t(Z_{s,t}^{x})+
\nabla_{\nabla_{\eta}Z_t^{s,x}}\nabla_{\nabla_{\eta^{\prime}}Z_t^{s,x}}\sigma_t(Z_t^{s,x})\big\}\d
W_t, ~~ t\geq s,\quad \gamma_{s,s}={\bf0_n}.
\end{equation*}
By the Doob submartingale inequality and the It\^o isometry, besides
the Gronwall inequality and \eqref{a2}, we get that
\beg{equation}\label{b11}\begin{split} \sup_{s\le t\le
T}\mathbb{E}|\gamma_{s,t}|^2&\leq 16T\|\nabla^2\sigma\|_{T,\infty}^2
\e^{288T^2
\|\nabla\sigma\|_{T,\infty}^4+2T\|\nabla\sigma\|_{T,\infty}^2}|\eta|^2|\eta^{\prime}|^2.
\end{split}\end{equation}
From \eqref{h0} and  the Markov property, we have
\begin{equation*}
\begin{split}
&  \nn_\eta P_{s,t}^0f(x) =\E\bigg(\frac{(
P_{\frac{t+s}{2},t}^0f)(Z^{s,x}_{\frac{t+s}{2}})}{(t-s)/2}\int_s^{\frac{t+s}{2}}\<\si^{-1}_r(Z_r^{s,x})\nn_\eta
Z_r^{s,x},\d W_r\>\bigg).
\end{split}
\end{equation*}
This further gives that
\begin{equation*}
\begin{split}
&\frac{1}{2}(\nn_{\eta'}\nn_\eta
P_{s,t}^0f)(x)\\
&=\E\bigg(\frac{(\nn_{\nn_{\eta'}Z^{s,x}_{\frac{t+s}{2}}}
P_{\frac{t+s}{2},t}^0f)(Z^{s,x}_{\frac{t+s}{2}})}{t-s}\int_s^{\frac{t+s}{2}}\<\si_r^{-1}(Z_r^{s,x})\nn_\eta
Z_r^{s,x},\d W_r\>\bigg)\\
&\quad+\E\bigg(\frac{(
P_{\frac{t+s}{2},t}^0f)(Z^{s,x}_{\frac{t+s}{2}})}{t-s}\int_s^{\frac{t+s}{2}}
\<(\nn_{\nn_{\eta'}Z^{s,x}_r} \si_r^{-1})(Z_r^{s,x})\nn_\eta
Z_r^{s,x},\d W_r\>\bigg)\\
&\quad+\E\bigg(\frac{(
P_{\frac{t+s}{2},t}^0f)(Z^{s,x}_{\frac{t+s}{2}})}{t-s}\int_s^{\frac{t+s}{2}}\<\si_r^{-1}(Z_r^{s,x})\nn_{\eta'}\nn_\eta
Z_r^{s,x},\d W_r\>\bigg).
\end{split}
\end{equation*}
Thus, applying  Cauchy-Schwartz's inequality, \cite[Theorem 7.1,
p.39]{Mao} and  It\^o's isometry and taking \eqref{b1}, \eqref{b7}
and \eqref{b11} into consideration, we derive  that
\begin{equation}\label{be}
\begin{split}
&|\nn_{\eta'}\nn_\eta P_{s,t}^0f|^2(x)\\&
\le12\bigg\{6\|\si^{-1}\|_{T,\8}^2\frac{\E|\nn
P_{\frac{t+s}{2},t}^0f|^2(Z^{s,x}_{\frac{t+s}{2}})}{(t-s)^{5/2}}\\
&\quad\times(\E|{\nn_{\eta'}Z^{s,x}_{\frac{t+s}{2}}}|^4)^{1/2}\Big(\int_s^{\frac{t+s}{2}}\E|\nn_\eta
Z_r^{s,x}|^4\d r\Big)^{1/2}\\
&\quad+\frac{P_{s,t}^{0}f^{2}(x)}{(t-s)^2}\|\nn\si^{-1}\|_{T,\8}^2
\int_s^{\frac{t+s}{2}}(\E|\nn_{\eta'}Z^{s,x}_r|^4)^{1/2}(\E|\nn_{\eta}Z^{s,x}_r|^4)^{1/2}\d
r\\
&\quad+\frac{P_{s,t}^{0}f^{2}(x)}{(t-s)^2}\|\si^{-1}\|_{T,\8}^2\int_s^{\frac{t+s}{2}}\E|\nn_{\eta'}\nn_\eta
Z_r^{s,x}|^2\d r\bigg\}\\
&\le\tt\Lambda_{T,\si}|\eta|^2|\eta^{\prime}|^2\ff{P_{s,t}^{0}f^{2}(x)}{(t-s)^2},
\end{split}
\end{equation}
where $\tt\Lambda_{T,\si}>0$ is defined as in \eqref{z2}.

\smallskip

Set $\tt f:=f-f(x)$ for   $f\in\B_b(\R^n)$ which verifies
\begin{equation}\label{a1}
|f(x)-f(y)|\le\phi(|x-y|),~~~~x,y\in\R^n,
\end{equation}
where $\phi:\R_+\mapsto\R_+$ is increasing and $\phi^2$ is concave.
For $f\in\B_b(\R^n)$ such that \eqref{a1}, \eqref{be} implies that
\begin{equation}\label{v1}
\begin{split}
|\nn_{\eta'}\nn_\eta P_{s,t}^0f|^2(x)=|\nn_{\eta'}\nn_\eta
P_{s,t}^0\tt f|^2(x)& \le
\frac{\tt\Lambda_{T,\si}|\eta|^{2}|\eta^{\prime}|^{2}}{(t-s)^2}\E
|f(Z^{s,x}_t)-f(x)|^2\\
&\le\frac{\tt\Lambda_{T,\si}|\eta|^{2}|\eta^{\prime}|^{2}}{(t-s)^2}\phi^2(\|\si\|_{T,\8}(t-s)^{1/2}),
\end{split}
\end{equation}
where in the second display  we have used that
\begin{equation*}
Z^{s,x}_t-x=\int_s^t\si_r(Z^{s,x}_r)\d W_r,
\end{equation*}
and utilized Jensen's inequality as well as It\^o's isometry.

Let $f_t=b_t+\nabla_{b_t}u_t^\ll$. For any
 $ \lambda\ge
9\pi\Lambda_{T,\si}^2\|b\|_{T,\8}^2+4(\|b\|_{T,\8}+\Lambda_{T,\si})^2,
$ note from  {\bf(ii)}, \eqref{b7} and \eqref{be} that
\begin{equation*}
\begin{split}
|f_t(x)-f_t(y)|
&\le(1+\|\nn u^\lambda\|_{T,\8})\phi(|x-y|)+\|b\|_{T,\8}\|\nn
u_t^\ll(x)-\nn u_t(y)\|{\bf \bf{1}}_{\{|x-y|\ge1\}}\\
&\quad+\|b\|_{T,\8}\|\nn u_t^\ll(x)-\nn u_t(y)\|{\bf{1}}_{\{|x-y|\le1\}}\\
&\le\ff{3}{2}\phi(|x-y|)+\|b\|_{T,\8} \ss{|x-y|}{\bf{1}}_{\{|x-y|\ge1\}}\\
&\quad+10\Big(\Lambda_{T,\si}+
 \ss{\tt\Lambda_{T,\si}}\Big)\|b\|_{T,\8}^2\ss{|x-y|}\ss{|x-y|}\log\Big(\e+\ff{1}{|x-y|}\Big)
{\bf{1}}_{\{|x-y|\le1\}}\\
&\le\Big\{3+2\|b\|_{T,\8}+
 28\Big(\Lambda_{T,\si}+
 \ss{\tt\Lambda_{T,\si}}\Big)\|b\|_{T,\8}^2\Big\}
 \tt\phi(|x-y|)
\end{split}
\end{equation*}
with $ \tt\phi(s):=\ss{\phi^2(s)+s},\,s\ge0, $ where in the second
inequality we have used \cite[Lemma 2.2 (1)]{W}, and that the
function $[0,1]\ni x\mapsto\ss x\log(\e+\ff{1}{x})$ is
non-decreasing. As a result, {\bf(iii)} follows from \eqref{v1}.
\end{proof}

\section{Proof of Theorem \ref{T1.1}}\label{sec3}

With Lemma \ref{L2.1} in hand, we now in a position to complete the

\smallskip
\noindent{\bf Proof of Theorem \ref{T1.1}.} Throughout the whole
proof, we assume $ \lambda\ge
9\pi\Lambda_{T,\si}^2\|b\|_{T,\8}^2+4(\|b\|_{T,\8}+\Lambda_{T,\si})^2$
so that {\bf(i)}-{\bf(iii)} in Lemma \ref{L2.1} hold. For any
$t\in[0,T]$, applying It\^{o}'s formula to $x+u_t^\ll(x),x\in\R^n$,
we deduce from \eqref{2.1} that
\begin{equation}\label{b5}
\beg{split}
X_t+u_t^\ll(X_t)=x+u_0^\ll(x)+\lambda\int_0^tu_s^\ll(X_s)\d
s+\int_0^t\{{\bf I_{n\times n}}+(\nabla
u_s^\ll)(\cdot)\}(X_s)\si_s(X_s)\d W_s,
\end{split}
\end{equation}
where ${\bf I_{n\times n}}$ is an $n\times n$ identity matrix, and
that
\begin{equation}\label{b6}
\beg{split}
Y_t+u_t^\ll(Y_t)
&=x+u_0^\ll(x)+\lambda\int_0^tu_s^\ll(Y_s)\d s
+\int_0^t\{{\bf I_{n\times n}}+(\nabla
u_s^\ll)(\cdot)\}(Y_s){\si_{s_\delta} (Y_{s_\delta})\d W_s}\\
&\quad+\int_0^t\{{\bf I_{n\times n}}+(\nabla u_s^\ll)(\cdot)\}(Y_s)\{b_{s_\delta}(Y_{s_\delta})-b_s(Y(s))\}\d s\\
&\quad+\frac{1}{2}\int_0^t\sum_{k,j}{\langle\{
(\sigma_{s_\delta}\sigma^{\ast}_{s_\delta})(Y_{s_\delta})-(\sigma_s\sigma^{\ast}_s)(
Y_s)\}e_{k},e_{j}\rangle}
(\nabla_{e_{k}}\nabla_{e_{j}}u_s^\ll)(Y_s)\d s.
\end{split}
\end{equation}
 For notational
simplicity, set
\begin{equation}\label{s7}
 M^\ll_t:=X_t-Y_t+u^\ll_t(X_t)-u^\ll_t(Y_t).
 \end{equation}  Using the elementary inequality: $ (a+b)^2\le
(1+\vv)(a^2+\vv^{-1}b^2)$ for $\vv,a,b>0, $
 we derive from {\bf(ii)} that
\begin{equation*}
\begin{split}
|X_t-Y_t|^2
&\le(1+\vv)(|M_t^\ll|^2+\vv^{-1}|u^\ll_t(X_t)-u^\ll_t(Y_t)|^2)\\
&\le(1+\vv)\Big(|M_t^\ll|^2+\ff{\vv^{-1}}{4}|X_t-Y_t|^2\Big).
\end{split}
\end{equation*}
In particular, taking $\vv=1$ leads to
\begin{equation*}
|X_t-Y_t|^2 \le\ff{1}{2}|X_t-Y_t|^2+2|M_t^\ll|^2.
\end{equation*}
As a consequence,
\begin{equation}\label{b4}
\E\Big(\sup_{0\le s\le t}|X_s-Y_s|^2\Big) \le4\E\Big(\sup_{0\le s\le
t}|M_s^\ll|^2\Big).
\end{equation}
In what follows,   our goal is to estimate the term on the right
hand side of \eqref{b4}. Observe from the definition of the
Hilbert-Schmidt norm that
\begin{equation}\label{h1}
\begin{split}
&\int_0^t\E\Big|\sum_{k,j}{\langle[
(\sigma_{s_\delta}\sigma^{\ast}_{s_\delta})(
Y_{s_\delta})-(\sigma_s\sigma^{\ast}_s)( Y_s)]e_{k},e_{j}\rangle}
(\nabla_{e_{k}}\nabla_{e_{j}}u_s^\ll)(Y_s)\Big|^2\d s\\
&\1_T\|\nn^2
u^\ll\|_{T,\8}^2\int_0^t\E\|(\sigma_{s_\delta}\sigma^{\ast}_{s_\delta})(
Y_{s_\delta})-(\sigma_s\sigma^{\ast}_s)( Y_s)\|^2_{\mathrm{HS}}\d s.
\end{split}
\end{equation}
Thus, by H\"older's inequality, Doob's submartingale inequality and
It\^o's isometry, it follows from \eqref{b5}, \eqref{b6} and
\eqref{h1} that
\begin{align*}
\E\Big(\sup_{0\le s\le t}|M_s^\ll|^2\Big)
&\le C_T\bigg\{\lambda^2\int_0^t\E|u_s^\ll(X_s)-u_s^\ll(Y_s)|^2\d s\\
&\quad+(1+\|\nabla u\|_{T,\8}^2)\int_0^t\E|b_{s_\delta}(Y_s)-b_{s_\delta}(Y_{s_\delta})|^2\d s\\
&\quad+(1+\|\nabla u\|_{T,\8}^2)\int_0^t\E|b_s(Y_s)-b_{s_\delta}(Y_s)|^2\d s\\
&\quad+\int_0^t\E\|\{(\nabla u_s^\ll)(X_s)-\nabla
u_s^\ll(Y_s)\}\si_s(X_s)\|^2_{\mathrm{HS}}\d s\\
&\quad+(1+\|\nabla
u\|_{T,\8}^2)\int_0^t\E\|\si_{s_\delta}(X_s)-\si_{s_\delta} (
Y_{s_\delta})\|^2_{\mathrm{HS}}\d s\\
&\quad+\|\nn^2 u^\ll\|_{T,\8}^2\int_0^t\E\|\{
\sigma_{s_\delta}(Y_s)-\sigma_{s_\delta}( Y_{s_\delta}
)\}\sigma^{\ast}_{s_\delta}(
Y_{s_\delta} )\|^2_{\mathrm{HS}}\d s\\
&\quad+\|\nn^2 u^\ll\|_{T,\8}^2\int_0^t\E\|
\sigma_s(Y_s)\{\sigma^{\ast}_{s_\delta}(Y_s)-\sigma^{\ast}_{s_\delta}(
Y_{s_\delta} )\}\|^2_{\mathrm{HS}}\d s\\
&\quad+(1+\|\nabla
u\|_{T,\8}^2)\int_0^t\E\|\si_s(X_s)-\si_{s_\delta} (
X_s)\|^2_{\mathrm{HS}}\d s\\
&\quad+\|\nn^2 u^\ll\|_{T,\8}^2\int_0^t\E\|
\sigma_s(Y_s)\{\sigma^{\ast}_s(Y_s)-\sigma^{\ast}_{s_\delta}(
Y_s )\}\|^2_{\mathrm{HS}}\d s\\
&\quad+\|\nn^2 u^\ll\|_{T,\8}^2\int_0^t\E\|\{
\sigma_s(Y_s)-\sigma_{s_\delta}( Y_s )\}\sigma^{\ast}_{s_\delta}(
Y_{s_\delta} )\|^2_{\mathrm{HS}}\d s\bigg\}\\
&=:\sum_{i=1}^{10}I_i(t)
 \end{align*}
for some constant $C_T>0.$ Also, applying H\"older's inequality and
It\^o's isometry, we deduce from ({\bf A1}) that
\begin{equation}\label{hui}
\E|Y_t- Y_{t_\delta}|^2\le \beta_T\dd
\end{equation}
for some constant $\bb_T\ge1.$ By Taylor's expansion, it is readily
to see that
\begin{equation}\label{s1}
I_1(t)+I_4(t)\1 \{\lambda^2\|\nn u^\ll\|_{T,\8}^2+\|\nn^2
u^\ll\|_{T,\8}^2\|\si\|_{T,\8}^2\}\int_0^t\E|X_s-Y_s|^2\d s.
\end{equation}
From ({\bf A3}), one has
\begin{equation}\label{s2}
I_3(t)+\sum_{i=8}^{10}I_i(t)\1_T\{1+\|\nn u^\ll\|_{T,\8}^2+\|\nn^2
u^\ll\|_{T,\8}^2\|\si\|_{T,\8}^2\}\phi(\ss\delta)^2.
\end{equation}
In view of ({\bf A2}), we derive that
\begin{equation}\label{s3}
\begin{split}
&I_2(t)+\sum_{i=5}^7I_i(t)\\&\1\{1+\|\nn
u^\ll\|_{T,\8}^2\}\int_0^t\E\phi(|Y_s-Y_{s_\delta}|)^2\d s\\
&\quad+\{1+\|\nn
u^\ll\|_{T,\8}^2\}\|\nabla\si\|_{T,\8}^2\int_0^t\E|X_s-Y_s|^2\d s\\
&\quad+\{1+\|\nn u^\ll\|_{T,\8}^2+\|\nn^2
u^\ll\|_{T,\8}^2\|\si\|_{T,\8}^2\}\|\nabla\si\|_{T,\8}^2\int_0^t\E|Y_s-Y_{s_\delta}|^2\d
s.
\end{split}
\end{equation}
Thus, taking \eqref{hui}-\eqref{s3} into account and applying
Jensen's inequality gives that
\begin{equation*}\beg{split}
\E\Big(\sup_{0\le s\le t}|M_s^\ll|^2\Big)&\1_T
C_{T,\si,\lambda}\{\delta+\phi(\bb_T\ss\delta)^2\}+ C_{T,\si,\lambda}\int_0^t\E|X_s-Y_s|^2\d s,
\end{split}\end{equation*}
where \begin{equation}\label{x1}
\begin{split}
C_{T,\si,\lambda}:=\{1+\|\nn \si\|_{T,\8}^2\}\Big\{\ff{5}{4}
+(1+\ll^2)\|\nn^2 u^\ll\|_{T,\8}^2\|\si\|_{T,\8}^2\Big\}.
\end{split}
\end{equation} Owing to
$\phi\in\D,$ we conclude that $\phi(0)=0,$ $\phi'>0$ and $\phi''<0$
so that, for any $c>0$ and $\dd\in(0,1),$
\begin{equation*}
\phi(c\dd)=\phi(0)+\phi'(\xi)c\dd\ge\phi'(c)c\dd
\end{equation*}
by recalling $\dd\in(0,1),$ where $\xi\in(0,c\dd).$ This further
implies that
\begin{equation*}\beg{split}
\E\Big(\sup_{0\le s\le t}|M_s^\ll|^2\Big)&\1_T
C_{T,\si,\lambda}\phi(\bb_T\ss\delta)^2+
C_{T,\si,\lambda}\int_0^t\E|X_s-Y_s|^2\d s.
\end{split}\end{equation*}
Substituting this into \eqref{b4} gives that
\begin{equation*}\beg{split}
\E\Big(\sup_{0\le s\le t}|X_s-Y_s|^2\Big)&\1_T
C_{T,\si,\lambda}\phi(\bb_T\ss\delta)^2+
C_{T,\si,\lambda}\int_0^t\E|X_s-Y_s|^2\d s.
\end{split}\end{equation*}
Thus,  Gronwall's inequality implies that there exists $\tt C_T>0$
such that
\begin{equation}\label{z4}\beg{split}
\E\Big(\sup_{0\le s\le t}|X_s-Y_s|^2\Big)&\le \tt C_T
C_{T,\si,\lambda}\e^{\tt C_T
C_{T,\si,\lambda}}\phi(\bb_T\ss\delta)^2.
\end{split}\end{equation}
So the desired assertion holds immediately.

\section{Proof of Theorem \ref{T1.2}}\label{sec4}
We shall adopt the   cut-off approach to finish the

\begin{proof}[{\bf Proof of Theorem \ref{T1.2}}] Take $\psi\in
C_b^\8(\R_+)$ such that $0\le\psi\le1$, $\psi(r)=1$ for $r\in[0,1]$
and $\psi(r)=0$ for $r\ge2$. For any $t\in[0,T]$ and $k\ge1$, set
define the  cut-off functions
\begin{equation*}
b^{(k)}_t(x)=b_t(x)\psi(|x|/k)~~~~~\mbox{and}~~~~~\si^{(k)}_t(x)=\si_t(\psi(|x|/k)x),~~~~x\in\R^n.
\end{equation*}
It is easy to see that $b^{(k)}$ and $\si^{(k)}$ satisfy {\bf(A1)}.
For fixed $k\ge1,$ consider the following SDE
\begin{equation}\label{b8}
\d X^{(k)}_t=b^{(k)}_t(X^{(k)}_t)\d t+\si^{(k)}_t(X^{(k)}_t)\d
W_t,~~~t>0,~~~X^{(k)}_0=X_0=x.
\end{equation}
The corresponding continuous-time EM of \eqref{b8} is defined by
\begin{equation}\label{h5}
\d Y_t^{(k)}=b_{t_\delta}^{(k)}( Y_{t_\delta}^{(k)})\d
t+\sigma_{t_\delta}^{(k)} ( Y_{t_\delta}^{(k)})\d
W_t,~~~t>0,~~~Y_0^{(k)}=X_0=x.
\end{equation}
Applying the BDG inequality, the H\"older inequality and the
Gronwall inequality, we   deduce from {\bf (A1')}  that
\begin{equation}\label{w}
\E\Big(\sup_{0\le t\le T}|X_t|^4\Big)+\E\Big(\sup_{0\le t\le
T}|Y_t|^4\Big)+\E\Big(\sup_{0\le t\le
T}|X_t^{(k)}|^4\Big)+\E\Big(\sup_{0\le t\le T}|Y_t^{(k)}|^4\Big)\le
C_T
\end{equation}
for some constant $C_T>0$. Note that
\begin{equation*}
\begin{split}
\E\Big(\sup_{0\le t\le T}|X_t-Y_t|^2\Big)&\le2\E\Big(\sup_{0\le t\le
T}|X_t-X_t^{(k)}|^2\Big)+2\E\Big(\sup_{0\le t\le
T}|X_t^{(k)}-Y_t^{(k)}|^2\Big)\\
&\quad+2\E\Big(\sup_{0\le t\le T}|Y_t-Y_t^{(k)}|^2\Big)\\
&=:I_1+I_2+I_3.
\end{split}
\end{equation*}
For the terms $I_1$ and $I_3$, in terms of the Chebyshev inequality
we find from \eqref{w} that
\begin{equation*}
\begin{split}
I_1+I_3&\1\E\Big(\sup_{0\le t\le T}|X_t-X_t^{(k)}|^2{\bf{1}}_{\{\sup_{0\le t\le T}|X_t|\ge k\}}\Big)\\
&\quad+\E\Big(\sup_{0\le t\le T}|Y_t-Y_t^{(k)}|^2{\bf{1}}_{\{\sup_{0\le t\le T}|Y_t|\ge k\}}\Big)\\
&\1\ss{\E\Big(\sup_{0\le t\le T}|X_t|^4\Big)+\E\Big(\sup_{0\le t\le
T}|X_t^{(k)}|^4\Big)}\frac{\ss{\E\Big(\sup_{0\le t\le
T}|X_t|^2\Big)}}{k}\\
&\quad+\ss{\E\Big(\sup_{0\le t\le T}|Y_t|^4\Big)+\E\Big(\sup_{0\le
t\le T}|Y_t^{(k)}|^4\Big)}\frac{\ss{\E\Big(\sup_{0\le t\le
T}|Y_t|^2\Big)}}{k}\\
&\1_T\frac{1}{k},
\end{split}
\end{equation*}
where in the first display we have used the facts that $\{X_t\neq
X_t^{(k)}\}\subset\{\sup_{0\le s\le t}|X_s|\ge k\}$ and $\{Y_t\neq
Y_t^{(k)}\}\subset\{\sup_{0\le s\le t}|Y_s|\ge k\}$. Observe from
{\bf(A1')} that $
9\pi\Lambda_{T,\si^{(k)}}^2\|b^{(k)}\|_{T,\8}^2+4(\|b^{(k)}\|_{T,\8}+\Lambda_{T,\si^{(k)}})^2\le
\e^{ck^2} $ for some $c>0.$ Next, according to \eqref{z4},
 by taking $\lambda=\e^{ck^2}$ there
exits $C_T>0$ such that
\begin{equation*}
\begin{split}
I_2\le \e^{C_T C_{T,\si^{(k)},\lambda}}\phi_k(\bb_T\ss\delta)^2.
\end{split}
\end{equation*}
Herein, $ C_{T,\si^{(k)},\lambda}>0$ is defined as in \eqref{x1}
with $\si$ and $u^\lambda$ replaced by $\si^{(k)}$ and
$u^{\lambda,k}$, respectively, where $u^{\ll,k}$ solves \eqref{b2}
by writing $b^{(k)}$ instead of $b$. Consequently, we conclude that
\begin{equation}\label{z6}
\E\Big(\sup_{0\le t\le T}|X_t-Y_t|^2\Big)\le\frac{\bar c_0}{k}+\bar
c_0\e^{C_T C_{T,\si^{(k)},\lambda}}\phi_k(\bb_T\ss\dd)^2
\end{equation}
for some $\bar c_0>0.$ For any $\vv>0$, taking $k=\ff{2\bar
c_0}{\vv}$ and letting $\dd$ go to zero   implies that
\begin{equation*}
\lim_{\dd\rightarrow0}\E\Big(\sup_{0\le t\le
T}|X_t-Y_t|^2\Big)\le\vv.
\end{equation*}
Thus, \eqref{h6} follows due to the arbitrariness of $\vv$.

 For
$\phi_k(s)=\e^{\e^{c_0k^4}}s^{\alpha},s\ge0,$ with
$\alpha\in(0,1/2],$ we deduce from Lemma \ref{L2.1} {\bf(iii)} that
\begin{equation}\label{z5}
 \|\nn^2 u^{\ll,k}\|_{T,\8}\le\ff{1}{2}
\end{equation}
whenever
\begin{equation}\label{z8}
\begin{split}
\lambda&\ge\Big\{2\Upsilon_{T,\si^{(k)}}\Big(\e^{\e^{c_0k^4}}\|\si^{(k)}\|_{T,\8}^\alpha
\Gamma(\alpha/2)+\|\si^{(k)}\|_{T,\8}^{1/2}\Gamma(1/4)\Big)\Big\}^{2/\alpha}\\
&\quad~~~~+9\pi(\Lambda_{T,\si^{(k)}})^2\|b^{(k)}\|_{T,\8}^2+4(\|b^{(k)}\|_{T,\8}+\Lambda_{T,\si^{(k)}})^2.
\end{split}
\end{equation}
Since the right hand side of \eqref{z8} can be bounded by $\e^{\bar
C_T\,k^4}$ for some constant $\bar C_T>0$ due to {\bf (A1')}, we can
take $\ll=\e^{\bar C_T\,k^4}$  so that \eqref{z5} holds. Thus,
\eqref{z6}, together with \eqref{z5} and {\bf(A1')}, yields that
\begin{equation*}
\E\Big(\sup_{0\le t\le T}|X_t-Y_t|^2\Big)\le\frac{\bar C_T}{k}+\bar
C_T\e^{\e^{ \tt C_Tk^4}}\dd^\aa
\end{equation*}
for some constants $\bar C_T,\tt C_T>0$. Thus, \eqref{z7} follows
immediately by taking
\begin{equation*}
k=(\tt C_T\log\log\dd^{-\aa\vv})^{\frac{1}{4}}.
\end{equation*}
\end{proof}

\section{Proof of Theorem \ref{th1}}\label{sec5}
 The proof of Theorem \ref{th1} relies on regularization properties of
the following $\R^{2n}$-valued degenerate parabolic equation
\begin{equation}\label{song}
\partial_t u_t^\lambda+\mathscr{L}_t^{b,\si} u_t^\lambda+b_t=\lambda
u_t^\lambda,~~~~u_T^\ll={\bf0_{2n}},~~~t\in[0,T],~~~\lambda>0,
\end{equation}
where ${\bf0_{2n}}$ is the zero vector in $\R^{2n}$,
\begin{equation*}
b_t:= \Spvek[c]{b_t^{(1)};b_t^{(2)}} ~\mbox{ and
}~\mathscr{L}_t^{b,\si}u^\ll:=\frac{1}{2}\sum_{i,j=1}^n{\langle
(\si_t\si_t^*)(\cdot)e_{i},e_{j}\rangle}\nabla_{e_{i}}^{(2)}\nabla_{e_{j}}^{(2)}u^\ll+
\nn_{b_t^{(1)}}^{(1)} u^\ll+\nn_{b_t^{(2)}}^{(2)} u^\ll.
\end{equation*}

For any $\phi\in\D_0\cap\mathscr{S}_0$, set
\begin{equation*}
\bar{\mathscr{Q}}_\phi:=\sup_{t\in[0,T]}\{[b_t^{(1)}]_{\phi_{[2/3]},\8}+\interleave\nn^{(2)}b_t^{(1)}\interleave_{\8,\phi}+\|\si_t^{-1}\|_\8
+\interleave\si_t^{-1}\interleave_{\phi_{[2/3]}}+\interleave
b_t^{(2)}\interleave_{\phi_{[2/3]},\phi}\}
\end{equation*}
and
\begin{equation*}
\mathscr{Q}_\phi:=\bar{\mathscr{Q}}_\phi+\sup_{t\in[0,T]}[b_t^{(2)}]_{\phi_{[2/3]},\phi^{7/2}},~~~~\mathscr{Q}_\phi^\prime:=\bar{\mathscr{Q}}_\phi
+\sup_{t\in[0,T]}[\nn^{(2)}\si_t]_{\phi_{[1/9]},\8},
\end{equation*}
where, for $\phi\in\mathscr{S}_0$,
\begin{equation*}
\phi_{[\alpha]}(t):=t^\alpha\phi(t){\bf{1}}_{\{t\ge1\}}+2c_\alpha
t{\bf{1}}_{\{t>1\}}~~\mbox{ with }~
~c_\alpha:=\sup_{s\in(0,1]}(s^\alpha\phi(s)).
\end{equation*}

The following key lemma on regularity estimate of solution to
\eqref{song} is cited from \cite[Theorem 2.3]{WZ} and is an
essential ingredient in analyzing numerical approximation.

\begin{lem}\label{regu}
{\rm Under {\bf (C1)}, \eqref{song} has a unique smooth solution
such that for all $t\in[0,T],$
\begin{equation}\label{song1}
\begin{split}
\interleave\nn
u_t^\ll\interleave_{1_{[1/3],\8}}&+\|\nn^{(1)}\nn^{(2)}u_t^\ll\|_\8+\interleave\nn^{(2)}\nn^{(2)}u_t^\ll\interleave_{\phi^{3/2}} \\
&\le
C\int_0^t\e^{-\lambda(t-s)}\frac{\phi((t-s)^{\ff{1}{2}})}{t-s}[b_s]_{\phi_{[2/3]},\phi^{7/2}}\d
s,
\end{split}
\end{equation}
and
\begin{equation}\label{song2}
\interleave\nn
u_t^\ll\interleave_{1_{[1/3],\8}}+\|\nn\nn^{(2)}u_t^\ll\|_\8\le
C^\prime\int_0^t\e^{-\lambda(t-s)}\frac{\phi((t-s)^{\ff{1}{2}})}{t-s}[b_s]_{\phi_{[2/3]},\phi}\d
s,
\end{equation}
where $C=C(\phi,\mathscr{Q}_\phi)$ and
$C^\prime=C^\prime(\phi,\mathscr{Q}_\phi^\prime)$ are
 increasing w.r.t. $\mathscr{Q}_\phi$ and
$\mathscr{Q}_\phi^\prime$, respectively.

}
\end{lem}

From now on, we move forward  to complete the
\begin{proof}[{\bf Proof of Theorem \ref{th1}}] For notational
simplicity, set
\[
X_t:=\Spvek[c]{X_t^{(1)};X_t^{(2)}},\quad
Y_t:=\Spvek[c]{Y_t^{(1)};Y_t^{(2)}} \mbox{ and } ~~
~b_t(x):=\Spvek[c]{b_t^{(1)}(x);b_t^{(2)}(x)},~~~x\in\R^{2n}.
\]
Then \eqref{b3} and \eqref{s0} can be reformulated respectively as
\begin{equation*}
\d X_t=b_t(X_t)\d t+\Spvek[c]{{\bf0_{n\times n}};\si_t}(X_t)\d
W_t,~~~t>0,~~~X_0=x=\Spvek[c]{x_1;x_2}\in\R^{2n},
\end{equation*}
where ${\bf0_{n\times n}}$ is an $n\times n$ zero matrix, and
\begin{equation*}
\d Y_t=b_{t_\dd}(Y_{t_\dd})\d t+\Spvek[c]{{\bf0_{n\times
n}};\si_{t_\dd}}(Y_{t_\dd})\d W_t,~~~t>0,~~~Y_0=x\in\R^{2n}.
\end{equation*}
Note from \eqref{song1} that there exists an $\ll_0>0$ sufficiently
large such that
\begin{equation}\label{s6}
 \|\nn u^\ll\|_{T,\8}+\|\nn^{(2)}\nn^{(2)}u^\ll\|_{T,\8}+\|\nn\nn^{(2)}u^\ll\|_{T,\8}\le
 \ff{1}{2},~~~~\ll\ge\ll_0.
 \end{equation}
Applying It\^o's formula to $x+u_t^\ll(x)$ for any $ x\in\R^{2n}$,
we deduce that
\begin{equation}\label{s4}
X_t+u_t^\ll(X_t)= x+u_0^\ll(x)+\lambda\int_0^tu_s^\ll(X_s)\d
s+\int_0^t\Spvek[c]{{\bf0_{n\times n}};\si_s}(X_s)\d
W_s+\int_0^t(\nn^{(2)}_{\si_s\d W_s}u_s^\ll)(X_s),
\end{equation}
and that
\begin{equation}\label{s5}
\begin{split}
Y_t+u_t^\ll(Y_t)&= x+u_0^\ll(x)+\lambda\int_0^tu_s^\ll(Y_s)\d s\\
&\quad+\int_0^t\{{\bf I_{2n\times 2n}}+(\nn
u_s)(\cdot)\}(Y_s)\{b_{s_\delta}(Y_{s_\delta})-b_s(Y_s)\}\d s\\
&\quad+\int_0^t\Spvek[c]{{\bf0_{n\times
n}};\si_{s_\dd}}(Y_{s_\dd})\d
W_s+\int_0^t(\nn^{(2)}_{\si_{s_\dd}(Y_{s_\dd})\d W_s}u_s^\ll)(Y_s)\\
&\quad+ \frac{1}{2}\int_0^t\sum_{k,j=1}^n{\langle\{
(\sigma_{s_\delta}\sigma^{\ast}_{s_\delta})(Y_{s_\delta})-(\sigma_s\sigma^{\ast}_s)(
Y_s)\}e_{k},e_{j}\rangle}
(\nabla_{e_{k}}^{(2)}\nabla_{e_{j}}^{(2)}u_s^\ll)(Y_s)\d s,
\end{split}
\end{equation}
where ${\bf I_{2n\times 2n}}$ is an $2n\times 2n$ identity matrix.
Thus,  using H\"older's inequality, Doob's sub-martingale inequality
and It\^o's isometry and taking \eqref{h1} into consideration gives
that
\begin{align*}
\E\Big(\sup_{0\le s\le t}|M_s^\ll|^2\Big)
&\le C_{0,T}\bigg\{ \int_0^t\E|u_s^\ll(X_s)-u_s^\ll(Y_s)|^2\d s\\
&\quad+(1+\|\nn u^\ll\|_{T,\8}^2)\int_0^t\E|b_{s_\delta}(Y_s)-b_{s_\delta}(Y_{s_\delta})|^2\d s\\
&\quad+(1+\|\nn u^\ll\|_{T,\8}^2)\int_0^t\E|b_s(Y_s)-b_{s_\delta}(Y_s)|^2\d s\\
&\quad+\int_0^t\E\|\{(\nabla^{(2)} u_s^\ll)(X_s)-\nabla^{(2)}
u_s^\ll(Y_s)\}\si_s(X_s)\|^2_{\mathrm{HS}}\d s\\
&\quad+(1+\|\nabla^{(2)}
u^\ll\|_{T,\8}^2)\int_0^t\E\|\si_{s_\dd}(X_s)-\si_{s_\delta} (
Y_{s_\delta})\|^2_{\mathrm{HS}}\d s\\
&\quad+(1+\|\nabla^{(2)}
u^\ll\|_{T,\8}^2)\int_0^t\E\|\si_s(X_s)-\si_{s_\delta} (
X_s)\|^2_{\mathrm{HS}}\d s\\
&\quad+\|\nabla^{(2)}\nabla^{(2)} u^\ll\|_{T,\8}^2\int_0^t\E\|\{
\sigma_{s_\delta}(Y_s)-\sigma_{s_\delta}( Y_{s_\delta}
)\}\sigma^{\ast}_{s_\delta}(
Y_{s_\delta} )\|^2_{\mathrm{HS}}\d s\\
&\quad+\|\nabla^{(2)}\nabla^{(2)} u^\ll\|_{T,\8}^2\int_0^t\E\|
\sigma_s(Y_s)\{\sigma^{\ast}_{s_\delta}(Y_s)-\sigma^{\ast}_{s_\delta}(
Y_{s_\delta} )\}\|^2_{\mathrm{HS}}\d s\\
&\quad+\|\nabla^{(2)}\nabla^{(2)} u^\ll\|_{T,\8}^2\int_0^t\E\|
\sigma_s(Y_s)\{\sigma^{\ast}_s(Y_s)-\sigma^{\ast}_{s_\delta}(
Y_s )\}\|^2_{\mathrm{HS}}\d s\\
&\quad+\|\nabla^{(2)}\nabla^{(2)}u^\ll\|_{T,\8}^2\int_0^t\E\|\{
\sigma_s(Y_s)-\sigma_{s_\delta}( Y_s )\}\sigma^{\ast}_{s_\delta}(
Y_{s_\delta} )\|^2_{\mathrm{HS}}\d s\bigg\}\\
&=:\sum_{i=1}^{10}J_i(t)
\end{align*}
for some constant $C_{0,T}>0$, where   $M_t^\ll$ is defined as in
\eqref{s7}. By using H\"older's inequality and \cite[Theorem 7.1,
p.39]{Mao}, {\bf(C1)} implies that
\begin{equation}\label{s9}
\E|Y_t-Y_{t_\dd}|^p\1\dd^{\ff{p}{2}},~~~~p\ge1.
\end{equation}
Utilizing Taylor's expansion, one gets from \eqref{hui} and
\eqref{s6} that
\begin{equation*}
\begin{split}
J_1(t)+J_4(t)+J_5(t)&\1\{1+\|\nn u^\ll\|_{T,\8}^2+\|\nn
\nn^{(2)}u^\ll\|_{T,\8}^2\|\si\|_{T,\8}^2\}\int_0^t\E|X_s-Y_s|^2\d
s\\
&\quad+\{1+\|\nn^{(2)}
u^\ll\|_{T,\8}^2\}\int_0^t\E|Y_s-Y_{s_\dd}|^2\d
s\\
&\1\dd+\int_0^t\E|X_s-Y_s|^2\d s.
\end{split}
\end{equation*}
Next, {\bf(C1)}, {\bf(C5)} and \eqref{s6} yield  that
\begin{equation*}
J_3(t)+J_6(t)+J_9(t)+J_{10}(t)\1\phi(\ss\dd)^2.
\end{equation*}
Additionally, by virtue of {\bf(C1)}, {\bf(C2)}, and \eqref{s6}, we
infer from {\bf(C3)}  that
\begin{equation*}
\begin{split}
J_2(t)+J_7(t)+J_8(t)&\1\dd+\int_0^t\E|b_{s_\delta}(Y_s^{(1)},Y_s^{(2)})-b_{s_\delta}(Y_{s_\delta}^{(1)},Y_s^{(2)})|^2\d
s\\
&\quad+
\int_0^t\E|b_{s_\delta}(Y_{s_\dd}^{(1)},Y_s^{(2)})-b_{s_\delta}(Y_{s_\delta}^{(1)},Y_{s_\delta}^{(2)})|^2\d
s\\
&\le
C_{1,T}\bigg\{\dd+\int_0^t\E|b_{s_\delta}^{(1)}(Y_s^{(1)},Y_s^{(2)})-b_{s_\delta}^{(1)}(Y_{s_\delta}^{(1)},Y_s^{(2)})|^2\d
s\\
&\quad+\int_0^t\E|b_{s_\delta}^{(2)}(Y_s^{(1)},Y_s^{(2)})-b_{s_\delta}^{(2)}(Y_{s_\delta}^{(1)},Y_s^{(2)})|^2\d
s\\
&\quad+
\int_0^t\E|b_{s_\delta}^{(1)}(Y_{s_\dd}^{(1)},Y_s^{(2)})-b_{s_\delta}^{(1)}(Y_{s_\delta}^{(1)},Y_{s_\delta}^{(2)})|^2\d
s\\
&\quad+
\int_0^t\E|b_{s_\delta}^{(2)}(Y_{s_\dd}^{(1)},Y_s^{(2)})-b_{s_\delta}^{(2)}(Y_{s_\delta}^{(1)},Y_{s_\delta}^{(2)})|^2\d
s\bigg\}\\
&=:C_{1,T}\dd+\sum_{i=1}^4\Lambda_i(t)
\end{split}
\end{equation*}
for some constant $C_{1,T}>0$. From {\bf(C2)}, {\bf(C3)}, \eqref{s9}
and $\phi\in\D^\vv$, we derive from H\"older's inequality and
Jensen's inequality that
\begin{equation}\label{s8}
\begin{split}
\Lambda_1(t)+\Lambda_2(t)&\1\sum_{i=1}^2\int_0^t
\E\bigg(\ff{|b_{s_\delta}^{(i)}(Y_s^{(1)},Y_s^{(2)})
-b_{s_\delta}^{(i)}(Y_{s_\delta}^{(1)},Y_s^{(2)})|}{|Y_s^{(1)}-Y_{s_\delta}^{(1)}|^{\ff{2}{3}}\phi(|Y_s^{(1)}-Y_{s_\delta}^{(1)}|)}\bf{1}_{\{Y_s^{(1)}\neq Y_{s_\delta}^{(1)}\}}\\
&\quad~~~~~~~~~~~~~\times
|Y_s^{(1)}-Y_{s_\delta}^{(1)}|^{\ff{2}{3}}\phi(|Y_s^{(1)}-Y_{s_\delta}^{(1)}|)\bigg)^2\d
s\\
&\1\int_0^t\E(|Y_s^{(1)}-Y_{s_\delta}^{(1)}|^{\ff{2}{3}}\phi(|Y_s^{(1)}-Y_{s_\delta}^{(1)}|))^2\d
s\\
&\1
\int_0^t\Big(\E\phi(|Y_s^{(1)}-Y_{s_\delta}^{(1)}|)^{2(1+\vv)}\Big)^{\ff{1}{1+\vv}}\Big(\E|Y_s^{(1)}-Y_{s_\delta}^{(1)}|^{\ff{4(1+\vv)}{3\vv}}\Big)^{\ff{\vv}{1+\vv}}\d
s\\
&\1\dd^{\ff{2}{3}} \phi(C_{2,T}\ss\dd)^2
\end{split}
\end{equation}
for some constant $C_{2,T}>0.$ With regard to the term
$\Lambda_3(t)$, {\bf(C1)} and \eqref{s9} leads to
\begin{equation}
\Lambda_3(t)\1\|\nn^{(2)}b^{(1)}\|_{T,\8}^2\int_0^t\E|Y_s^{(1)}-Y_{s_\delta}^{(1)}|^2\d
s\1\dd.
\end{equation}
Since $[b_t^{(2)}]_{\8,\phi}<\8$ due to {\bf(C3)}, observe from
Jensen's inequality and \eqref{s9} that
\begin{equation*}
\begin{split}
\Lambda_4(t)&\1\int_0^t\E\bigg(\ff{|b_{s_\delta}^{(2)}(Y_{s_\dd}^{(1)},Y_s^{(2)})-b_{s_\delta}^{(2)}
(Y_{s_\delta}^{(1)},Y_{s_\delta}^{(2)})|}{\phi(|Y_s^{(2)}-Y_{s_\delta}^{(2)}|)}{\bf{1}}_{\{Y_s^{(2)}\neq
Y_{s_\delta}^{(2)}\}}\times\phi(|Y_s^{(2)}-Y_{s_\delta}^{(2)}|)\bigg)^2\d
s\\
&\1\int_0^t\E\phi(|Y_s^{(2)}-Y_{s_\delta}^{(2)}|)^2\d s\\
&\1\phi(C_{3,T}\ss\dd)^2
\end{split}
\end{equation*}
for some constant $ C_{3,T}>0.$ Consequently, we arrive at
\begin{equation*}
\begin{split}
\E\Big(\sup_{0\le s\le t}|X_s-Y_s|^2\Big)\1_T  \phi(C_{4,T}\ss\dd)^2
+\int_0^t\E|X(s)-Y(s)|^2\d s
\end{split}
\end{equation*}
for some constant $C_{4,T}\ge1$.  Thus, the desired assertion
follows from the Gronwall inequality.

\end{proof}

\paragraph{Acknowledgement.} The authors would like to thank Professor Feng-Yu Wang for helpful comments.

\beg{thebibliography}{99} {\small

\setlength{\baselineskip}{0.14in}
\parskip=0pt

\bibitem{B16} Bachmann, S., Well-posedness and stability for a class of
stochastic delay differential equations with singular drift,
arXiv:1608.07534v1.

\bibitem{BY}  Bao, J., Yuan, C., Convergence rate of EM scheme for
SDDEs, {\it Proc. Amer. Math. Soc.}, {\bf141} (2013),  3231--3243.

\bibitem{BGT} Bingham, N.~H.,  Goldie, C.~M.,  Teugels, J.~L., Regular Variation,
Cambridge University Press, Cambridge, UK, 1987.

\bibitem{C12} Chaudru de Raynal, P.~E., Strong existence and uniqueness for
stochastic differential equation with H\"older drift and degenerate
noise, arXiv:1205.6688v3.

\bibitem{GW} Guillin, A., Wang, F.-Y., Degenerate Fokker-Planck equations:
 Bismut formula, gradient estimate and Harnack inequality, {\it J. Differential Equations}, {\bf253} (2012),   20--40.

\bibitem{G98}  Gy\"ongy, I., A note on Euler's approximations, {\it Potential Anal.}, {\bf8} (1998),  205--216.

\bibitem{GR}  Gy\"ngy, I., R\'{a}sonyi, M., A note on Euler approximations for SDEs with H\"older continuous diffusion
coefficients,
 {\it Stochastic Process. Appl.}, {\bf121} (2011),  2189--2200.

\bibitem{GS}Gy\"ongy, I., Sabanis, S., A note on Euler approximations for
stochastic differential equations with delay, {\it Appl. Math.
Optim.}, {\bf68} (2013), 391--412.

\bibitem{HK}
Halidias, N., Kloeden, P. E.,  A note on the Euler-Maruyama scheme
for stochastic differential equations with a discontinuous monotone
drift coefficient, {\it BIT}, {\bf48}(2008), 51--59.

\bibitem{Klo}
Kloeden, P. E., Platen, E.,   {\sl Numerical solution of stochastic
differential equations,}  Springer, 1992, Berlin.

\bibitem{Mao}
Mao, X.,  \emph{Stochastic differential equations and applications,}
Second Edition, Horwood Publishing Limited, 2008, Chichester.

\bibitem{MSH}   Mattingly, J.~C., Stuart, A.~M., Higham, D.~J., Ergodicity for SDEs and approximations: locally Lipschitz vector
fields and degenerate noise,
 {\it Stochastic Process. Appl.}, {\bf101} (2002),   185--232.

\bibitem{MX} Mikulevicius, R., Xu,   F., On the rate of convergence of strong Euler approximation for SDEs driven by L\'{e}vy
processes,  arXiv:1608.02303v1.

\bibitem{NT} Ngo, H.-L., Taguchi, D., Strong rate of convergence for the Euler-Maruyama approximation of stochastic differential equations
with irregular coefficients,
 {\it Math. Comp.}, {\bf85} (2016),  1793--1819.

\bibitem{NT2} Ngo, H.-L., Taguchi, D., Strong convergence for the Euler-Maruyama
approximation of stochastic differential equations with
discontinuous coefficients, arXiv:1604.01174v1.

\bibitem{NT1} Ngo, H.-L., Taguchi, D., On the Euler-Maruyama approximation for one-dimensional
stochastic differential equations with irregular coefficients,
arXiv:1509.06532v1.

\bibitem{PT} Pamen, O. M., Taguchi, D., Strong rate of convergence for the Euler-Maruyama
approximation of SDEs with H\"older continuous drift coefficient,
arXiv1508.07513v1.

\bibitem{W14}   Wang, F.-Y.,  Hypercontractivity for Stochastic Hamiltonian
Systems,  arXiv:1409.1995.

\bibitem{W}  Wang, F.-Y.,  Gradient estimates and applications for SDEs in Hilbert space with multiplicative noise and Dini
continuous drift,  {\it J. Differential Equations}, {\bf260} (2016),
2792--2829.

\bibitem{WZ}Wang, F.-Y., Zhang, X., Degenerate SDE with H\"older-Dini
Drift and Non-Lipschitz Noise Coefficient, {\it SIAM J. Math.
Anal.}, {\bf48} (2016),   2189--2226.

\bibitem{WZ1}Wang, F.-Y., Zhang, X.,  Degenerate SDEs in Hilbert spaces with
rough drifts, {\it Infin. Dimens. Anal. Quantum Probab. Relat.
Top.}, {\bf18} (2015),   1550026, 25 pp.

\bibitem{WZ13} Wang, F.-Y., Zhang, X., Derivative formula and applications for degenerate diffusion
semigroups,
 {\it J. Math. Pures Appl.},  {\bf99} (2013),   726--740.

\bibitem{Y02}   Yan, L., The Euler scheme with irregular coefficients, {\it Ann.
Probab.}, {\bf30} (2002),  1172--1194.

\bibitem{YM} Yuan, C., Mao, X., A note on the rate of convergence of the Euler-Maruyama method for stochastic differential
equations,
 {\it Stoch. Anal. Appl.}, {\bf26} (2008),  325--333.

\bibitem{Zhang}   Zhang, X., Stochastic flows and Bismut formulas for stochastic
Hamiltonian systems, {\it Stoch. Proc. Appl.}, {\bf 120} (2010),
1929--1949.

}
\end{thebibliography}

\end{document}